\documentclass[3p,number]{elsarticle}

\title{Homotopy invariance of 4-manifold decompositions: connected sums}
\author{Qayum Khan}
\ead{qkhan@indiana.edu}
\address{Department of Mathematics, University of Notre Dame, Notre Dame, IN 46556 U.S.A.}

\DeclareMathAlphabet{\matheurm}{U}{eur}{m}{n}

\usepackage{amsthm,amsmath,amssymb}
\usepackage{amscd,graphicx}

\newtheorem{thm}{Theorem}[section]
\newtheorem{cor}[thm]{Corollary}
\newtheorem{lem}[thm]{Lemma}

\theoremstyle{definition}
\newtheorem{defn}[thm]{Definition}
\newtheorem{rem}[thm]{Remark}
\newtheorem{exm}[thm]{Example}

\newcommand{\CP}{\mathbb{CP}}
\newcommand{\E}{\mathbb{E}}
\newcommand{\bbH}{\mathbb{H}}
\newcommand{\bL}{\mathbb{L}}
\newcommand{\R}{\mathbb{R}}
\newcommand{\RP}{\mathbb{RP}}
\newcommand{\Z}{\mathbb{Z}}

\newcommand{\cN}{\mathcal{N}}
\newcommand{\cS}{\mathcal{S}}

\newcommand{\DIFF}{\mathrm{DIFF}}
\newcommand{\TOP}{\mathrm{TOP}}

\newcommand{\infdec}{{-\infty}}

\newcommand{\cdim}{\mathrm{cdim}}
\newcommand{\Cok}{\mathrm{Cok}}
\newcommand{\hofiber}{\mathrm{hofiber}}

\newcommand{\id}{\mathrm{id}}

\newcommand{\intr}{\mathrm{int}}
\newcommand{\Isom}{\mathrm{Isom}}
\newcommand{\Ker}{\mathrm{Ker}}

\newcommand{\Nil}{\mathrm{Nil}}

\newcommand{\rel}{~\mathrm{rel}~}

\newcommand{\UNil}{\mathrm{UNil}}
\newcommand{\Wh}{\mathrm{Wh}}

\newcommand{\bdry}{\partial}
\newcommand{\longra}{\longrightarrow}
\newcommand{\ol}[1]{\overline{#1}}
\newcommand{\xra}[1]{\xrightarrow{#1}}

\begin{document}

\begin{abstract}
Given any homotopy equivalence $f: M \to X_1 \# \cdots \# X_n$ of closed orientable 4-manifolds, where each fundamental group $\pi_1(X_i)$ satisfies Freedman's Null Disc Lemma, we show that $M$ is topologically $h$-cobordant to a connected sum $M' = M'_1 \# \cdots \# M'_n$ such that $f$ is $h$-bordant to some $f'_1 \# \cdots \# f'_n$ with each $f'_i: M'_i \to X_i$ a homotopy equivalence. Moreover, such a replacement $M'$ of $M$ is unique up to a connected sum of $h$-cobordisms. In summary, the existence and uniqueness, up to $h$-cobordism, of connected sum decompositions of such orientable 4-manifolds $M$ is an invariant of homotopy equivalence.

Also we establish that the Borel Conjecture is true in dimension 4, up to $s$-cobordism, if the fundamental group satisfies the Farrell--Jones Conjecture.
\end{abstract}

\begin{keyword}
4-manifold \sep connected sum \sep homotopy equivalence \sep $h$-cobordism \sep Borel Conjecture
\end{keyword}

\maketitle
\section{Introduction}

\subsection{Homotopy invariance of connected sums---stable version}

For simplicity, we begin with the stable version of our main result (Theorem~\ref{thm:main}).  This version follows easily from a recent algebraic calculation of $\UNil$ for free products of groups by F.~Connolly and J.~Davis \cite{CD2} and from an earlier development of stable geometric topology by S.~Cappell and J.~Shaneson \cite{CS1, Cappell_split}.

\begin{thm}\label{thm:main_stable}
Let $X$ be a compact connected orientable topological manifold of dimension 4.
Suppose the fundamental group $\pi_1(X)$ is a free product of groups $\Gamma_1, \ldots, \Gamma_n$. Then there exist compact connected topological 4-manifolds $X_1, \ldots, X_n$ with each fundamental group $\pi_1(X_i)$ isomorphic to $\Gamma_i$  such that there is a bijection between $(S^2 \times S^2)$-stable $h$-structure sets:
\[
\#: \prod_{i=1}^n \ol{\cS}_\TOP^h(X_i) \longra \ol{\cS}_\TOP^h(X).
\]
Moreover, these $X_i$ are unique up to $(S^2 \times S^2)$-stabilization and re-ordering.
\end{thm}

\begin{proof}
By the stable prime decomposition of Kreck--L\"uck--Teichner \cite{KLT1}, there exist 4-manifolds $X_i$, unique up to stabilization and permutation, with fundamental groups $\Gamma_i$ such that $X$ is $(S^2 \times S^2)$-stably homeomorphic to $X_1 \# \cdots \# X_n$. By theorems of Waldhausen \cite{Waldhausen_Rings} and Connolly--Davis \cite{CD2}, the algebraic $K$- and $L$-theory splitting obstruction groups associated to each connecting 3-sphere vanish:
\[
\widetilde{\Nil}_0=0 \;\text{ and }\; \UNil_5^h = 0.
\]
Therefore, by the equivalence of Theorem~\ref{thm:generalized_weinberger}(2), using Cappell's high-dimensional splitting theorem \cite{Cappell_free, Cappell_split}, we obtain inductively that $\#$ is a bijection.
\end{proof}

\subsection{Homotopy invariance of connected sums---unstable version}

Our Main Theorem (Thm.~\ref{thm:main}) is phrased technically in terms of the classes $NDL$ and $SES^h_+$, which we define below.  The difficulty in the proof is in observing new extensions of the geometric topology developed by S.~Cappell \cite{Cappell_split} and S.~Weinberger \cite{Weinberger_fibering}.

\begin{defn}[Freedman]
A discrete group $G$ is \textbf{$NDL$ (or good)} if the $\pi_1$-Null Disc Lemma holds for it (see \cite{FreedmanTeichner} for details). The class $NDL$ is closed under the operations of forming subgroups, extensions, and filtered colimits.
\end{defn}

This class contains subexponential and exponential growth \cite{FQ, FreedmanTeichner, KQ}.

\begin{thm}[Freedman--Quinn, Freedman--Teichner, Krushkal--Quinn]
The class $NDL$ contains all virtually polycyclic groups and all groups of subexponential growth.
\end{thm}

\begin{exm}
Here are some exotic examples in $NDL$. The semidirect product $\Z^2 \rtimes_\alpha \Z$ with $\alpha=\left(\begin{smallmatrix}2 & 1 \\ 1 & 1\end{smallmatrix}\right)$ is polycyclic but has exponential growth. For all integers $n \neq -1, 0, 1$, the Baumslag--Solitar group $BS(1,n) = \Z[1/n] \rtimes_n \Z$ is finitely presented and solvable but not polycyclic. Grigorchuk's infinite 2-group $G$ is finitely generated but not finitely presented and has intermediate growth.
\end{exm}

Recall that, unless specified in the notation, the structure sets $\cS_\TOP^h$ and normal invariants $\cN_\TOP$ are homeomorphisms on the boundary (that is, rel $\bdry$) \cite[\S6.2]{HillmanBook}.

\begin{defn}
Let $Z$ be a non-empty compact connected topological 4-manifold. Denote the fundamental group $\pi := \pi_1(Z)$ and orientation character $\omega := w_1(\tau_Z)$. We declare that \textbf{$Z$ has class $SES^h$} if there exists an exact sequence of based sets:
\[
\cN_\TOP(Z \times I) \xrightarrow{~\sigma_5^h~} L_5^h(\pi,\omega) \xrightarrow{~\bdry~} \cS_\TOP^h(Z) \xrightarrow{~\eta~} \cN_\TOP(Z) \xrightarrow{~\sigma_4^h~} L_4^h(\pi,\omega).
\]
\end{defn}

The subclass $SES^h_+$ includes actions of groups in $K$- and $L$-theory (Defn.~\ref{defn:SESplus}).

This exact sequence has been proven for the above groups \cite[Thm.~11.3A]{FQ}.

\begin{thm}[Freedman--Quinn]\label{thm:FTKQ}
Let $X$ be a compact connected topological manifold of dimension 4. If $\pi_1(X)$ has class $NDL$, then $X$ has class $SES^h_+$ and satisfies the $s$-cobordism conjecture (i.e., all $s$-cobordisms on $X$ are homeomorphic to $X \times I$).
\end{thm}

Here is the Main Theorem of the paper. The existence and uniqueness question posed in the Title and Abstract, up to $h$-cobordism, is quantified in $\#$ of Part (2).

\begin{thm}\label{thm:main}
Let $X$ be a compact connected topological manifold of dimension 4.
\begin{enumerate}
\item
Suppose the fundamental group $\pi_1(X)$ is a free product of groups of class $NDL$. If $X$ is non-orientable, assume $\pi_1(X)$ is 2-torsionfree. Then there exists $r \geq 0$ such that the $r$-th stabilization $X \# r(S^2 \times S^2)$ has class $SES^h_+$.

\item
Suppose $X$ has the homotopy type of a connected sum $X_1 \# \cdots \# X_n$ such that each $X_i$ has class $SES^h_+$.  If $X$ is non-orientable, assume that $\pi_1(X)$ is 2-torsionfree. Then the homotopy connected sum $X$ has class $SES^h_+$. Moreover, the following induced function is a bijection:
\[
\#: \prod_{i=1}^n \cS_\TOP^h(X_i) \longra \cS_\TOP^h(X).
\]
\end{enumerate}
\end{thm}


The proof of our theorem consists of two steps: first homology split along each essential 3-sphere \cite{Weinberger_fibering}, and then perform a neck exchange trick \cite{FQ} to replace homology 3-spheres with genuine ones (cf.~\cite{KLT2, JK_RP4}). The first step is possible because the high-dimensional splitting obstruction group \cite{Cappell_free} has recently been shown to vanish \cite{CD2}. No direct surgeries are performed---only cobordisms are attached. Our techniques do not show triviality of $s$-cobordisms.

Indeed, it turns out that a limited form of surgery does work for free groups.

\begin{exm}
Suppose $X$ is a closed connected topological 4-manifold with free fundamental group: $\pi_1(X)=F_n$. Then a fixed stabilization $X \# r(S^2 \times S^2)$ has a  topological $s$-cobordism surgery sequence, for some $r \geq 0$ depending on $X$.
\end{exm}

Here are other caveats, which place our main theorem into historical context.

\begin{rem}
A homotopy decomposition into a connected sum need not exist. A counterexample to the homotopy Kneser conjecture with $\pi_1(X) = G_3 * G_5$ where $G_p := C_p \times C_p$ was constructed by M.~Kreck, W.~L\"uck, and P.~Teichner \cite{KLT2}.
\end{rem}

\begin{rem}
Given a homotopy decomposition into a connected sum, a homeomorphism decomposition need not exist. There exist infinitely many examples of non-orientable closed topological 4-manifolds homotopy equivalent to a connected sum ($X = \RP^4 \# \RP^4$) that are not homeomorphic a non-trivial connected sum \cite{JK_RP4, BDK}. Hence $\#$ is not always a bijection in the case $\pi_1(X) = D_\infty \in NDL$.
\end{rem}

\begin{rem}
For certain groups $\pi_1(X)$ unknown in $NDL$, such as poly-surface groups, results on exactness at $\cN_\TOP(X)$ are found in \cite{HillmanBook, Khan_smoothable, HegRep, CavSpa}.
\end{rem}

\begin{rem}
The modular group $PSL(2,\Z) \cong C_2 * C_3$ is an example of a free product of $NDL$ groups. It has a discrete cofinite-area action on $\bbH^2$. However our theorem in the non-orientable case excludes it and $SL(2,\Z) \cong C_4 *_{C_2} C_6$. The group $PSL(2,\Z)$ plays a key role in the orientable case of free products \cite{CD2}.
\end{rem}

Let us conclude this subsection with an application to fibering of 5-manifolds. Partial results were obtained in \cite{Weinberger_fibering, Khan_fibering}. The proof is located in Section~\ref{sec:main_proofs}.

\begin{thm}\label{thm:fibering}
Let $M$ be a closed topological 5-manifold. Let $X$ be a closed topological 4-manifold of class $SES^h_+$. Suppose $f: M \to S^1$ is a continuous map such that the induced infinite cyclic cover $\overline{M} = \hofiber(f)$ is homotopy equivalent to $X$. If the Farrell--Siebenmann fibering obstruction $\tau(f) \in \Wh_1(\pi_1 M)$ vanishes, then $f$ is homotopic to a topological $s$-block bundle projection with pseudofiber $X$.
\end{thm}

Note we obtain a fiber bundle projection if $X$ satisfies the $s$-cobordism conjecture.

\subsection{Topological $s$-rigidity for 4-dimensional manifolds}

The purpose of this final subsection is an elementary observation (Thm.~\ref{thm:rigidity}) from which we conclude the Borel Conjecture is true in dimension 4 up to topological $s$-cobordism, given that the fundamental group satisfies the Farrell--Jones Conjecture (see Cor.~\ref{cor:algebraic_rigidity}).

\begin{defn}
A compact topological manifold $Z$ is \textbf{topologically rigid} if, for all compact topological manifolds $M$, any homotopy equivalence $h: M \to Z$, with restriction $\bdry h: \bdry M \to \bdry Z$ a homeomorphism, is homotopic to a homeomorphism.
\end{defn}

Recall the Borel conjecture is proven for certain good groups \cite[Thm.~11.5]{FQ}.

\begin{thm}[Freedman--Quinn]
Suppose $Z$ is an aspherical compact topological 4-manifold such that $\pi_1(Z)$ is virtually polycyclic.  Then $Z$ is topologically rigid.
\end{thm}

The following crystallographic examples include the 4-torus $T^4$. It turns out that there are only finitely many examples in any dimension (e.g., see~\cite[Thm.~21]{Farkas}).

\begin{exm}
Suppose $\Gamma$ is a Bieberbach group of rank 4, that is, a torsionfree lattice in the Lie group $\Isom(\E^4)$. Then $Z = \R^4/\Gamma$ is topologically rigid (cf.~\cite{FH}).
\end{exm}

Let us now turn our attention to a weaker form of rigidity for general groups.

\begin{defn}
A compact topological manifold $Z$ is \textbf{topologically $s$-rigid} if, for all compact topological manifolds $M$, any homotopy equivalence $h: M \to Z$, with restriction $\bdry h: \bdry M \to \bdry Z$ a homeomorphism, is itself topologically $s$-bordant rel $\bdry M$ to a homeomorphism. It suffices that the Whitehead group $\Wh_1(\pi_1 Z)$ vanishes and the topological $s$-cobordism structure set $\cS_\TOP^s(Z)$ is a singleton.
\end{defn}

The following important basic observation does not seem to have appeared in the literature.
In particular, we do not assume that the fundamental group is $NDL$.

\begin{thm}\label{thm:rigidity}
Let $Z$ be a compact topological 4-manifold with fundamental group $\pi$ and orientation character $\omega: \pi \to \{\pm 1\}$. Suppose the surgery obstruction map $\sigma_4^s: \cN_\TOP(Z) \to L_4^s(\pi,\omega)$ is injective, and suppose the surgery obstruction map $\sigma_5^s: \cN_\TOP(Z \times I) \to L_5^s(\pi,\omega)$ is surjective. If $\Wh_1(\pi)=0$ then $Z$ is topologically $s$-rigid. Also $Z$ has class $SES^s_+$.
\end{thm}

We sharpen an observation of J.~Hillman \cite[Lem.~6.10]{HillmanBook} to include map data.

\begin{cor}\label{cor:unstable_rigidity}
Let $Z$ be a compact topological 4-manifold.  Suppose the product $Z \times S^1$ is topologically rigid. If $\Wh_1(\pi_1 Z)=0$ then $Z$ is topologically $s$-rigid.
\end{cor}

This allows us to generalize a theorem of J.~Hillman for surface bundles over surfaces \cite[Thm.~6.15]{HillmanBook}. His conclusion was that the source and target are abstractly $s$-cobordant. Our new feature is $s$-rigidity of the homotopy equivalence.

\begin{exm}
Suppose $Z$ is a compact topological 4-manifold that is the total space of a topological fiber bundle of aspherical surfaces over an aspherical surface. Then $Z$ is topologically $s$-rigid, as follows. By \cite[Theorem~6.2]{HillmanBook}, the group $\Wh_1(\pi_1 Z)$ vanishes. By the proof of \cite[Theorem~6.15]{HillmanBook}, the set $\cS_\TOP^s(Z \times S^1)$ is a singleton. Now apply Corollary~\ref{cor:unstable_rigidity}. Alternatively, we can use Corollary~\ref{cor:algebraic_rigidity} and the recently established validity of $FJ_L$ for polysurface groups \cite{BFL_lattice}.
\end{exm}

In the topology of high-dimensional manifolds, the following class of fundamental groups has been of intense interdisciplinary interest for at least the past two decades.

\begin{defn}
Denote $FJ_L$ as the class of groups $\Gamma$ that are $K$-flat and satisfy the Farrell--Jones Conjecture in $L$-theory \cite{FJ_iso}. That is, the elements $\Gamma$ of $FJ_L$ satisfy $\Wh_1(\Gamma \times \Z^n) = 0$ and $H_n^\Gamma(E_{\matheurm{all}}\Gamma, E_{\matheurm{vc}}\Gamma; \underline{\bL}_\Z^\infdec) = 0$ for all $n \geq 0$ (see~\cite{DL1}).
\end{defn}

We shall focus on the torsionfree case. This has nice subclasses \cite{FJ_GLmR, BL_CAT0, BFL_lattice}.

\begin{thm}[Farrell--Jones, Bartels--L\"uck, Bartels--Farrell--L\"uck]
Let $\Gamma$ be a discrete torsionfree group. Then $\Gamma$ has class $FJ_L$ if:
\begin{itemize}
\item $\Gamma$ is the fundamental group of a complete $A$-regular Riemannian manifold with all sectional curvatures non-positive, or

\item $\Gamma$ is hyperbolic with respect to the word metric, or

\item $\Gamma$ admits a cocompact proper action by isometries on a complete finite-dimensional $\mathrm{CAT}(0)$ metric space, or

\item $\Gamma$ is a virtually polycyclic group (equivalently, a virtually poly-$\Z$ group), or

\item $\Gamma$ is a cocompact lattice in a virtually connected Lie group.
\end{itemize}
\end{thm}

We state our $s$-cobordism answer to the Borel conjecture for exponential growth.

\begin{cor}\label{cor:algebraic_rigidity}
Suppose $Z$ is an aspherical compact topological 4-manifold such that $\pi_1(Z)$ has class $FJ_L$. Then $Z$ is topologically $s$-rigid. Also $Z$ has class $SES^h_+$.
\end{cor}

\begin{exm}
Topological $s$-rigidity occurs if $Z-\bdry Z$ is complete finite-volume hyperbolic. That is, $Z - \bdry Z = \R^4/\Gamma$ for some torsionfree lattice $\Gamma$ in $\Isom(\bbH^4)$.
\end{exm}

\begin{exm}
A non-Riemannian example of topological $s$-rigidity is the closed 4-manifold $Z$ of M.~Davis \cite{DavisM_aspherical}.  The universal cover $\widetilde{Z}$ is a complete  $\mathrm{CAT}(0)$ metric space. Most strikingly, $\widetilde{Z}$ is contractible but not homeomorphic to $\R^4$.
\end{exm}

The next example involves multiple citations, so we give a formal proof later. Currently, due to $\Nil$ summands, it is unknown if its Whitehead group vanishes.

\begin{cor}\label{cor:mappingtorus}
Suppose $Z$ is the mapping torus of a homeomorphism of an aspherical closed 3-manifold $K$. If $\Wh_1(\pi_1 Z)=0$ then $Z$ is topologically $s$-rigid.
\end{cor}

Now, let us pass to connected sums, which fail to be aspherical if there is more than one factor. The next statement shall follow from Theorems~\ref{thm:main} and \ref{thm:rigidity}. Below, we write $\cdim(G)$ for the cohomological dimension of any discrete group $G$.

\begin{cor}\label{cor:connectsum_vanishingsecondhomotopy}
Let $n > 0$. For each $1 \leq i \leq n$, let $X_i$ be a compact oriented topological 4-manifold. If each fundamental group $\Gamma_i := \pi_1(X_i)$ is torsionfree  of class $FJ_L$ with $\cdim(\Gamma_i) \leq 4$, and each mod-two second homotopy group vanishes: $\pi_2(X_i) \otimes \Z_2 = 0$, then the connected sum $X := X_1 \# \cdots \# X_n$ is topologically $s$-rigid.
\end{cor}

Next, we illustrate the basic but important example of non-aspherical oriented factors $X_i=S^1 \times S^3$. Here, the connected sum $X$ has free fundamental group $F_n$.

\begin{exm}
Let $n>0$. Recall $\Wh_1(\Z)=0$. Then, by Corollary~\ref{cor:connectsum_vanishingsecondhomotopy}, the closed 4-manifold $X = \# n(S^1 \times S^3)$ has class $SES^h_+$ and is topologically $s$-rigid.
\end{exm}

Finally, we specialize Corollary~\ref{cor:connectsum_vanishingsecondhomotopy} to the setting of the Borel conjecture.

\begin{cor}
Let $n > 0$. For each $1 \leq i \leq n$, suppose $X_i$ is an aspherical compact oriented topological 4-manifold with fundamental group $\Gamma_i := \pi_1(X_i)$ of class $FJ_L$. Then the connected sum $X := X_1 \# \cdots \# X_n$ is topologically $s$-rigid.\qed
\end{cor}

Here is an outline of the rest of the paper.  Foundations are laid in Sections~\ref{sec:structuresets}--\ref{sec:Weinberger}, where we expand work of Cappell and Weinberger in dimension four.
Applications are made in Sections~\ref{sec:main_proofs}--\ref{sec:rigidity_proofs}, where we prove the stated results of the Introduction. The reader may find most of our notation and terminology in Kirby--Taylor \cite{KT}.

\subsection*{Acknowledgments}

The author thanks Masayuki Yamasaki for his kind hospitality.
This paper was conceived in May 2009 at the Okayama University of Science.
Jonathan Hillman and Andrew Ranicki provided expert e-mail consultation.
Finally, the author is grateful to the Hausdorff Research Insititute for Mathematics.
Through discussions there in October 2009 at the Rigidity Trimester Program, the earlier results on topological $s$-rigidity were extended to the class of $FJ_L$ summands.
Research funding was provided in part by NSF grants DMS-0353640 (RTG) and DMS-0904276.
The referee is appreciated for several clarifying comments.

\section{The language of structure sets}\label{sec:structuresets}

To start, the following equivalence relations play prominent roles in Section~\ref{sec:Weinberger}.

\begin{defn}
Let $Z$ be a topological space.
Let $M, M'$ be compact topological manifolds.
Let $h: M \to Z$ and $h': M' \to Z$ be continuous maps.
A \textbf{bordism $H: h \to h'$ rel $\bdry$} is a compact topological cobordism $(W;M,M')$ rel $\bdry$ and a continuous map $|H|: W \to Z \times I$ such that $H|_M=h$ and $H|_{M'}=h'$.
We call $H: h \to h'$ a \textbf{$h$-bordism rel $\bdry$} (resp.~\textbf{$s$-bordism rel $\bdry$}) if $(W;M,M')$ is an $h$-cobordism (resp.~$s$-cobordism).
\end{defn}

Next, we relativize the surgical language in the Introduction (cf.~\cite{Wall}).

\begin{defn}
Let $Z$ be a topological manifold such that the boundary $\bdry Z$ is collared. Let $\bdry_0 Z$ be a union of components of $\bdry Z$.  The pair $(Z,\bdry_0 Z)$ is called a \textbf{$\TOP$ manifold pair}. Write $\bdry_1 Z := \bdry Z - \bdry_0 Z$. The induced triple $(Z;\bdry_0 Z,\bdry_1 Z)$ is an example of a \textbf{$\TOP$ manifold triad} (in other words, a cobordism).
\end{defn}

Here is the precise definition of the relative structure set that we use in proofs.

\begin{defn}
Let $(Z,\bdry_0 Z)$ be a compact $\TOP$ 4-manifold pair.
Write $\Gamma_0 := \pi_1(\bdry_0 Z)$, the fundamental groupoid of $\bdry_0 Z$.
The \textbf{structure set $\cS_\TOP^h(Z,\bdry_0 Z)$} consists of $\sim$-equivalence classes of continuous maps $(h; \bdry_0 h, \bdry_1 h): (M;\bdry_0 M,\bdry_1 M) \to (Z;\bdry_0 Z, \bdry_1 Z)$ of compact $\TOP$ 4-manifolds triads such that:
\begin{itemize}
\item
 $h: M \to Z$ is a homotopy equivalence,

\item
$\bdry_0 h: \bdry_0 M \to \bdry_0 Z$ is a $\Z[\Gamma_0]$-homology equivalence, and

\item
$\bdry_1 h: \bdry_1 M \to \bdry_1 Z$ is a homeomorphism.
\end{itemize}
We call such $(h, \bdry_0 h): (M, \bdry_0 M) \to (Z, \bdry_0 Z)$ a \textbf{homotopy--homology equivalence}.
Here, $h \sim h'$ if there exists a $\TOP$ bordism $(H;\bdry_0 H, \bdry_1 H): (h;\bdry_0 h,\bdry_1 h) \to (h';\bdry_0 h', \bdry_1 h')$ such that:
\begin{itemize}
\item $H: W \to Z \times I$ is a homotopy equivalence,

\item
$\bdry_0 H: \bdry_0 W \to \bdry_0 Z \times I$ is a $\Z[\Gamma_0]$-homology equivalence, where $(\bdry_0 W; \bdry_0 M, \bdry_0 M')$ is a cobordism, and

\item
$\bdry_1 H: \bdry_1 W \to \bdry_1 Z \times I$ is a homeomorphism, where $(\bdry_1 W; \bdry_1 M, \bdry_1 M')$ is a cobordism.
\end{itemize}
We call such $(H,\bdry_0 H): (h,\bdry_0 h) \to (h',\bdry_0 h')$ a \textbf{homotopy--homology $h$-bordism}.
\end{defn}

The 4-dimensional relative surgery sequence is defined carefully as follows.
It is an $h$-version of Wall's sequence (middle of \cite[p.~115]{Wall}) with homotopy equivalences to $\bdry_0 Z$ and with homeomorphisms to $\bdry_1 Z$.

\begin{defn}
Let $(Z,\bdry_0 Z)$ be a compact $\TOP$ 4-manifold pair. Denote the fundamental groupoids $\Gamma := \pi_1(Z)$ and $\Gamma_0 := \pi_1(\bdry_0 Z)$. Denote the orientation character $\omega := w_1(\tau_Z): \Gamma \to \{\pm 1\}$. We declare that \textbf{$(Z,\bdry_0 Z)$ has class $SES^h$} if there exists an exact sequence of based sets:
\[
\cN_\TOP(Z \times I, \bdry_0 Z \times I) \xrightarrow{~\sigma_5^h~} L_5^h(\Gamma,\Gamma_0,\omega) \xrightarrow{~\bdry~} \cS_\TOP^h(Z, \bdry_0 Z) \xrightarrow{~\eta~} \cN_\TOP(Z, \bdry_0 Z) \xrightarrow{~\sigma_4^h~} L_4^h(\Gamma,\Gamma_0,\omega).
\]
\end{defn}

Last is the enhancement to include actions of certain groups in $K$- and $L$-theory.

\begin{defn}\label{defn:SESplus}
In addition, we declare that \textbf{$(Z,\bdry_0 Z)$ has class $SES^h_+$} if, for all elements $h \in \cS_\TOP^h(Z,\bdry_0 Z)$ and $t \in \Wh_1(\Gamma)$ and $x \in L_5^h(\Gamma,\Gamma_0,\omega)$, there exist:
\begin{itemize}
\item an action of the group $\Wh_1(\Gamma)$ on the set $\cS_\TOP^h(Z,\bdry_0 Z)$ such that:
\begin{itemize}
\item
there is an $h$-bordism $F: W \to Z \times I$ rel $\bdry$ from $h: M \to Z$ to $t(h): M' \to Z$ with Whitehead torsion $\tau(W;M,M')=t$, and
\end{itemize}

\item an action of the group $L_5^h(\Gamma,\Gamma_0,\omega)$ on the set $\cS_\TOP^h(Z,\bdry_0 Z)$ such that:
\begin{itemize}
\item there exists a normal bordism $F$ from $h$ to $x(h)$ with $\sigma_5^h(F) = x$, and
\item the equation $\bdry(x) = x(\id_Z)$ holds.
\end{itemize}
\end{itemize}
\end{defn}

Before moving on, we consider the stable version of the above structure set.

\begin{defn}
Let $(Z,\bdry_0 Z)$ be a compact $\TOP$ 4-manifold pair.
The \textbf{stable structure set $\ol{\cS}_\TOP^h(Z,\bdry_0 Z)$} consists of $\ol{\sim}$-equivalence classes of homotopy--homology equivalences $h: (M, \bdry_0 M) \to (Z\# r(S^2 \times S^2), \bdry_0 Z)$ for any $r \geq 0$.
Here, we define $h \ol{\sim} h'$ if there exist $s, s' \geq 0$ and a homotopy--homology $h$-bordism $H: h \# \id_{s(S^2 \times S^2)} \to h' \# \id_{s'(S^2 \times S^2)}$.
\end{defn}

The next theorem was proven by S.~Cappell and J.~Shaneson \cite{CS1} (cf.~\cite{KT}) and reformulated here.

\begin{thm}[Cappell--Shaneson]\label{thm:CappellShaneson}
Let $(Z,\bdry_0 Z)$ be a compact $\TOP$ 4-manifold pair. Denote the fundamental groupoids $\Gamma := \pi_1(Z)$ and $\Gamma_0 := \pi_1(\bdry_0 Z)$ and orientation character $\omega: \Gamma \to \{\pm 1\}$.  Then there is an exact sequence of based sets:
\[
\cN_\TOP(Z \times I, \bdry_0 Z \times I) \xrightarrow{~\sigma_5^h~} L_5^h(\Gamma,\Gamma_0,\omega) \xrightarrow{~\bdry~} \ol{\cS}_\TOP^h(Z, \bdry_0 Z) \xrightarrow{~\eta~} \cN_\TOP(Z, \bdry_0 Z) \xrightarrow{~\sigma_4^h~} L_4^h(\Gamma,\Gamma_0,\omega).
\]
The group $L_5^h(\Gamma,\Gamma_0,\omega)$ acts on the set $\ol{\cS}^h_\TOP(Z,\bdry_0 Z)$ in such a way that the above map $\bdry$ is equivariant.
\end{thm}

\section{A Weinberger-type homology splitting theorem}\label{sec:Weinberger}

Now we are ready to improve the $\Lambda$-splitting theorem of S.~Weinberger \cite{Weinberger_fibering} by  slightly modifying his proof.  In essence Theorems~\ref{thm:main} \& \ref{thm:main_stable} shall be its corollaries.

\begin{defn}
In the setting below, the homotopy equivalence $h: M \to X$ is \textbf{$\Z[\Gamma_0]$-split} if $h$ is topologically transverse to $X_0$ and its restriction $h_0: h^{-1}(X_0) \to X_0$ is a $\Z[\Gamma_0]$-homology equivalence (hence $h-h_0: h^{-1}(X-X_0) \to X-X_0$ is also).
\end{defn}

\begin{thm}\label{thm:generalized_weinberger}
Let $X$ be a non-empty compact connected topological 4-manifold. Let $X_0$ be a closed connected incompressible separating topological 3-submanifold of $X$. The decomposition of manifolds $X = X_1 \cup_{X_0} X_2$ induces the decomposition of fundamental groups $\Gamma = \Gamma_1 *_{\Gamma_0} \Gamma_2$. Define a closed simply-connected 8-manifold
\[
Q := \CP^4 \# (S^3 \times S^5) \# (S^3 \times S^5).
\]
Let $M$ be a compact topological 4-manifold. Suppose $h: M \to X$ is a homotopy equivalence such that the restriction $\bdry h: \bdry M \to \bdry X$ is a homeomorphism.
\begin{enumerate}
\item
Assume $(*)$: the group $\Gamma_0$ has class $NDL$ and the 4-manifold pairs $(X_1,X_0)$ and $(X_2,X_0)$ have class $SES^h_+$.
Then $h$ is topologically $s$-bordant rel $\bdry M$ to a homotopy equivalence $h''': M''' \to X$ $\Z[\Gamma_0]$-split along $X_0$ if and only if $h \times \id_Q$ is homotopic rel $\bdry M \times Q$ to a homotopy equivalence split along $X_0 \times Q$.

\item
Do not assume Hypothesis $(*)$. Then, for some $r \geq 0$, the $r$-th stabilization $h \# \id_{r(S^2 \times S^2)}$ is homotopic rel $\bdry M$ to a homotopy equivalence $h''': M''' \to X\# r(S^2 \times S^2)$ $\Z[\Gamma_0]$-split along $X_0$ if and only if $h \times \id_Q$ is homotopic rel $\bdry M \times Q$ to a homotopy equivalence split along $X_0 \times Q$.
\end{enumerate}
Moreover, there is an analogous statement if $X_0$ is two-sided and non-separating.
\end{thm}

Note the map $\Gamma_0 \to \Gamma$ is injective, but the amalgam $\Gamma$ need not have class $NDL$. Observe the 8-manifold $Q$ has both Euler characteristic and signature equal to one.

\begin{cor}[Weinberger]\label{cor:weinberger}
In the previous theorem, instead of $(*)$, assume $(**)$: $\bdry X$ is empty and the fundamental group $\Gamma$ has class $NDL$. Then $h$ is homotopic to a $\Z[\Gamma_0]$-split homotopy equivalence along $X_0$ if and only if $h \times \id_Q$ is homotopic to a split homotopy equivalence along $X_0 \times Q$.
\end{cor}

\begin{proof}
Since $\Gamma$ has class $NDL$, the subgroups $\Gamma_0, \Gamma_1, \Gamma_2$ have class $NDL$. Then, since $\Gamma_0, \Gamma_1, \Gamma_2$ have class $NDL$, by \cite{FreedmanTeichner, KQ}, the 4-manifold pairs $(X_i, X_0)$ have class $SES^h_+$. Hence Hypothesis $(**)$ implies Hypothesis $(*)$.
Now, since $\Gamma \in NDL$, by \cite{FreedmanTeichner, KQ}, the $\TOP$ $s$-cobordism of Theorem \ref{thm:generalized_weinberger}(1) is a product.
\end{proof}

\begin{rem}
Weinberger's theorem (Cor.~\ref{cor:weinberger}) \cite[Thm.~1]{Weinberger_fibering} was stated in a limited form. The only applicable situations were injective amalgamated products $\Gamma = \Gamma_1 *_{\Gamma_0} \Gamma_0 = \Gamma_1$ and $\Gamma = C_2 * C_2 = D_\infty$ in class $NDL$. (The second case was applied in \cite{JK_RP4, BDK}.) We effectively delete the last phrase in his proof. Earlier, there was a homology splitting result of M.~Freedman and L.~Taylor \cite{FreedmanTaylor} which required that $\Gamma = \Gamma_0 *_{\Gamma_0} \Gamma_0 = \Gamma_0$ but did not require that $\Gamma$ have class $NDL$.
\end{rem}

Next, we modify Weinberger's clever cobordism argument, adding a few details. We suppress the orientation characters $\omega$ needed in the non-orientable case.

\begin{proof}[Proof of Theorem~\ref{thm:generalized_weinberger}(1)]
For brevity, we sometimes denote ${}^Q$ for either $\times Q$ or $\times \id_Q$.

$(\Longrightarrow)$: Since $\dim(X^Q)=12 >4$, this follows from two high-dimensional facts.
By the $\TOP$ $s$-cobordism theorem \cite{Solopko}, the product of any 5-dimensional $s$-cobordism with $Q$ is homeomorphic to a product.
By the handlebody version of Quillen's plus construction (for example, see \cite[\S 11.2]{FQ}; the high-dimensional $\TOP$ version can be extracted from \cite[Annex 3, \S6--\S9]{KS}), the product of any 4-dimensional $\Z[\Gamma_0]$-split equivalence with $\id_Q$ can be exchanged along 2- and 3-handles in $M''' \times Q$ to become a split homotopy equivalence.

$(\Longleftarrow)$: Suppose $h^Q$ is homotopic to a homotopy equivalence split along $X_0^Q$. By $\TOP$ transversality \cite{FQ}, we may assume, up to homotopy rel $\bdry M$, that $h: M \to X$ is $\TOP$ transverse to $X_0$. There is an induced decomposition of compact manifolds $M = M_1 \cup_{M_0} M_2$, where for all $j=0,1,2$ the restrictions $h_j := h|M_j: M_j \to X_j$ are degree one $\TOP$ normal maps and $\bdry h_j$ are homeomorphisms.

Since $Q$ has signature equal to one, by the periodicity and product formulas \cite[\S8]{RanickiATS2}, for each $i=1,2$, the following relative surgery obstruction vanishes:
\[
\sigma_*(h_i,h_0) ~\cong~ \sigma_*(h_i,h_0) \otimes \sigma^*(Q) ~=~ \sigma_*(h_i^Q,h_0^Q) ~=~ 0 ~\in~ L_{12}^h(\Gamma,\Gamma_0).
\]
Then, by exactness at $\cN_\TOP$ in Hypothesis $(*)$, for each $i=1,2$, there exists a $\TOP$ normal bordism $(F_i,\bdry_0 F_i): (W_i,\bdry W_i) \to (X_i,X_0)$ from $(h_i,h_0): (M_i,M_0) \to (X_i,X_0)$ to a homotopy--homology equivalence $(h'_i,\bdry h'_i): (M'_i,\bdry M'_i) \to (X_i,X_0)$.
Note that the 3-manifolds $\bdry M'_1$ and $\bdry M'_2$ may not be homeomorphic.

We take three steps to construct an $s$-cobordism from $h$ to an $h'''$.
Figure~\ref{fig:Weinberger} illustrates the first step:
\begin{figure}[h!]\centering
\includegraphics[scale=0.50]{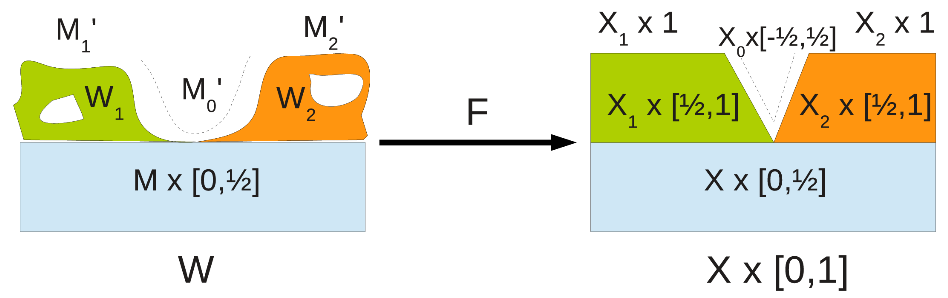}
\caption{relabeled version of Weinberger's \cite[Figure~1]{Weinberger_fibering}}
\label{fig:Weinberger}
\end{figure}

The precise, set-theoretic definitions are as follows:
\begin{eqnarray*}
F ~:=~ F_1 \cup_{h_1} (h \times \id_{[0,\frac{1}{2}]}) \cup_{h_2} F_2 ~,
& 
& X' ~:=~ (X_1 \times 1) \sqcup (X \times 0) \sqcup (X_2 \times 1)\\
W ~:=~ W_1 \cup_{M_1} (M \times {\textstyle[0,\frac{1}{2}]}) \cup_{M_2} W_2 ~,
&  M'_0 ~:=~ \bdry_0 W_1 \cup_{M_0} \bdry_0 W_2 ~,
& M' ~:=~ M'_1 \sqcup (M \times 0) \sqcup M'_2.
\end{eqnarray*}
Observe that $(F,\bdry_0 F): (W,M'_0) \to (X \times [0,1],X_0 \times [-\frac{1}{2},\frac{1}{2}])$ is a $\TOP$ normal map of manifold pairs, and that the restriction $\bdry_1 F: M' \to X'$ is a homotopy equivalence.

Next, the second step is to leech off surgery obstructions of the two halves of $F$ by attaching cobordisms.
Select a homotopy $H: M^Q \times [-1,0] \to X^Q$ to $h^Q$ from a homotopy equivalence $g = g_1 \cup_{g_0} g_2$ split along $X_0^Q$.
By $\TOP$ transversality \cite{FQ}, assume $H$ and $F^Q$ are transverse to $X_0^Q$.
Define $\TOP$ normal maps
\begin{eqnarray*}
G_0 ~:=~ H_0 \cup_{(h_0^Q \times 0)} (h_0^Q \times {\textstyle[0,\frac{1}{2}]}) ~,
& \quad & G_i ~:=~ H_i \cup_{(h_i^Q \times 0)} (h_i^Q \times {\textstyle[0,\frac{1}{2}]}) \cup_{(h_i^Q \times {\textstyle\frac{1}{2}})} F_i^Q ~.
\end{eqnarray*}
Note $H \cup_h F^Q = G_1 \cup_{G_0} G_2$.
Observe the restriction $\bdry_1 G_i = g_i \sqcup h_i'$ is a homotopy equivalence
and the complement $\bdry_0 G_i = G_0 \cup_{(h_0^Q \times {\textstyle{\frac{1}{2}}})} \bdry_0 F_i^Q$ is a normal map.
So there are defined surgery obstructions
\begin{eqnarray*}
x ~:=~ \sigma_*(F,\bdry_0 F) \in L_5^h(\Gamma,\Gamma_0) ~,
& \quad & x_i ~:=~ \sigma_*(G_i, \bdry_0 G_i) \in L_{13}^h(\Gamma_i,\Gamma_0).
\end{eqnarray*}
Denote the inclusion-induced homomorphism $j_i: L_5^h(\Gamma_i,\Gamma_0) \to L_5^h(\Gamma,\Gamma_0)$.
Then, by periodicity with $Q$, the cobordism invariance of surgery obstructions, and Wall's $\pi$-$\pi$ theorem \cite{Wall} (here, $L_*^h(\Gamma_0,\Gamma_0)=0$), we obtain:
\[
x ~\cong~ \sigma_*(F,\bdry_0 F) \otimes \sigma^*(Q) = \sigma_*(F^Q, \bdry_0 F^Q) ~=~ \sigma_*(H \cup_h F^Q, \bdry_0 F^Q) ~=~ j_1(x_1) + j_2(x_2).
\]
In particular, since $Q$ has Euler characteristic equal to one, we obtain that $x \in L_5^h(\Gamma,\Gamma_0)$ is the image of a surgery obstruction $x^B \in L_5^B(\Gamma,\Gamma_0)$ uniquely determined by $F^Q$, where the decoration subgroup is
\[
B ~:=~ j_1\Wh_1(\Gamma_1) + j_2\Wh_1(\Gamma_2) ~\subseteq~ \Wh_1(\Gamma).
\]

By existence of an $L_5^h$-action in Hypothesis $(*)$, for each $i=1,2$, there exists a $\TOP$ normal bordism $(F'_i,\bdry_0 F'_i): (W'_i,\bdry_0 W'_i) \to (X_i,X_0)$ from $(h'_i,\bdry h'_i)$ to $(h''_i, \bdry h''_i)$ with surgery obstruction $\sigma_*(F'_i,\bdry F'_i) = -x_i$ such that $(h''_i, \bdry h''_i): (M''_i, \bdry M''_i) \to (X_i, X_0)$ is a homotopy--homology equivalence.
Define:
\begin{eqnarray*}
F' ~:=~ F'_1 \cup_{h'_1} F \cup_{h'_2} F'_2 ~,
& \quad M''_0 ~:=~ \bdry_0 W'_1 \cup_{\bdry M'_1} M'_0 \cup_{\bdry M'_2} \bdry_0 W'_2 ~, \quad
& M'' ~:=~ M''_1 \sqcup (M \times 0) \sqcup M''_2.
\end{eqnarray*}
Observe $(F',\bdry_0 F'): (W',M_0'') \to (X \times [0,1], X_0 \times {\textstyle[-\frac{1}{2},\frac{1}{2}]})$ is a $\TOP$ normal map of pairs, and the complement $\bdry_1 F': M'' \to X'$ is a homotopy equivalence.
So there is defined a surgery obstruction which vanishes:
\[
\sigma_*(F',\bdry_0 F') ~=~ j_1(-x_1)+x+j_2(-x_2) ~=~ 0 ~\in~ L_5^B(\Gamma,\Gamma_0).
\]
Since the Null Disc Lemma holds for $\Gamma_0$, by 5-dimensional relative surgery \cite{Wall, FQ}, there is a normal bordism $G$ rel $M''$ to a $B$-torsion homotopy equivalence of pairs:
\[
(F'',\bdry_0 F'') : (W'',M'''_0) \longra (X \times [0,1], X_0 \times {\textstyle[-\frac{1}{2},\frac{1}{2}]}).
\]
In particular, $\bdry_1 F''=\bdry_1 F'$ restricts to a $\Z[\Gamma_0]$-homology equivalence $\bdry h'_1: \bdry M'_1 \to X_0$.
Hence $F''$ is a $B$-torsion $\TOP$ $h$-bordism from $h$ to a homotopy equivalence $\bdry_+ F'': \bdry_+W'' \to X$ $\Z[\Gamma_0]$-split along $X_0$.

Finally, the third step is to leech off the torsion obstructions of the two halves of the $h$-bordism $F''$.
Consider its Whitehead torsion
\[
t ~:=~ \tau(M \hookrightarrow W'') ~\in~ B.
\]
Then there exist $t_i \in \Wh(\Gamma_i)$ such that $t = j_1(t_1) + j_2(t_2)$.
By existence of a $\Wh_1$-action in Hypothesis $(*)$, for all $i=1,2$, there exist $h$-bordisms $F''_i \rel \bdry$ such that torsion of the domain $h$-cobordism is $j_i(-t_i)$. Therefore, by the sum formula, attaching these $h$-cobordisms to the top of $W''$ produces a $\TOP$ $s$-bordism $F''' := F''_1 \cup F'' \cup F''_2$ rel $\bdry M$ from $h$ to a homotopy equivalence $h''' := \bdry_+F'''$ $\Z[\Gamma_0]$-split along $X_0$.
\end{proof}

\begin{proof}[Proof of Theorem~\ref{thm:generalized_weinberger}(2)]
The argument in the stable case, Part (2), is similar to the unstable case, Part (1). The places where we used the hypothesis that $(X_1, X_0)$ and $(X_2, X_0)$ have class $SES^h_+$ can be replaced with the use of Theorem~\ref{thm:CappellShaneson}. Moreover, the places where we used the hypothesis that $\Gamma_0$ has class $NDL$ had target $X_0 \times I$ for surgery problems, and so can be replaced with the use of Theorem~\ref{thm:CappellShaneson}.

Realization of elements of the Whitehead group by $h$-cobordisms on any given compact 4-manifold is the same as in high dimensions \cite[p.~90]{RS}. Finally, by \cite[Theorem~9.1]{FQ}, any $\TOP$ $s$-cobordism $W'''$ on a compact 4-manifold admits a $\TOP$ handlebody structure. Then we proceed as in the proof of the high-dimensional $s$-cobordism theorem (e.g., see \cite[Thm.~6.19]{RS}), except we resolve double-point singularities of immersed Whitney 2-discs via Norman tricks \cite[Lem.~1]{Norman}. We conclude, for some $r \geq 0$, that the sum stabilization $W''' \natural r(S^2 \times S^2 \times I)$ (defined on \cite[p.~107]{FQ}) is homeomorphic to the product $(X \# r(S^2 \times S^2)) \times I$.
\end{proof}

\section{Proofs for the surgery sequence}\label{sec:main_proofs}

Again, we suppress the orientation characters $\omega$ used in the non-orientable case.
We start with a puncturing lemma. Section~\ref{sec:Weinberger} contains the terminology for pairs.

\begin{lem}\label{lem:puncture}
Let $Z$ be a non-empty compact connected topological 4-manifold. Write $pZ := Z - \intr\,D^4$. If $Z$ has class $SES^h_+$, then $(pZ, S^3)$ has class $SES^h_+$.
\end{lem}

\begin{proof}
Denote the fundamental group $\Gamma := \pi_1(Z)$. First, let $(M,\bdry_0 M)$ be a compact topological 4-manifold pair, and let $(f,\bdry_0 f): (M,\bdry_0 M) \to (pZ,S^3)$ be a degree-one $\TOP$ normal map of pairs that restricts to a homeomorphism $\bdry_1 f: \bdry_1 M \to \bdry Z$. Suppose the relative surgery obstruction vanishes: $\sigma_4^h(f)=0 \in L_4^h(\Gamma,1)$. Recall the geometric exact sequence of C.T.C.~Wall \cite[Cor.~3.1.1]{Wall}:
\[
\Z = L_4^h(1) \xra{~\varepsilon~} L_4^h(\Gamma) \longra L_4^h(\Gamma,1) \longra L_3^h(1) = 0.
\]
Then $\bdry_0 f: \bdry_0 M \to S^3$ is $\DIFF$ normally bordant to a $\Z$-homology equivalence $g: \Sigma \to S^3$. Since any closed oriented 3-manifold $\Sigma$ is parallelizable, by a theorem of M.~Freedman \cite[Cor.~9.3C]{FQ}, it follows there exists a $\TOP$ normal null-bordism of $g$ over $D^4$. Thus $(f,\bdry_0 f)$ is $\TOP$ normally bordant, as a pair relative to $\bdry_1 M$, to a degree-one map $f': M' \to Z$ such that $\bdry f': \bdry_1 M \to \bdry Z$ is a homeomorphism. Moreover, by connecting sum with copies of the $\TOP$ $E_8$-manifold or its reverse, we may assume that the absolute surgery obstruction vanishes: $\sigma_4^h(f') = 0 \in L_4^h(\Gamma)$. By hypothesis, $f'$ is $\TOP$ normally bordant to a homotopy equivalence $h: M'' \to Z$. We may assume that $h$ is transverse to a point $z \in Z$ and that $h^{-1}\{z\}$ is a singleton. Thus $(f,\bdry_0 f)$ is normally bordant to a homotopy equivalence $(ph,\id): (pM'',S^3) \to (pZ,S^3)$.
Therefore we obtain exactness at the normal invariants $\cN_\TOP(pZ,S^3)$.

Next, define an appropriate action of $L_5^h(\Gamma,1)$ on $\cS_\TOP^h(pZ,S^3)$ as follows. By puncturing at a transversal singleton $\{z\} \subset Z$ with connected preimage, we obtain a function $p: \cS_\TOP^h(Z) \to\cS_\TOP^h(pZ,S^3)$. By the existence of 1-connected $\TOP$ $h$-cobordism from a homology 3-sphere $\Sigma$ to the genuine one \cite[Cor.~9.3C]{FQ}, it follows that $p$ is surjective. By the topological plus construction \cite[Thm.~11.1A]{FQ}, applied to any homology $h$-cobordism of $S^3$ to itself, it follows that $p$ is injective. By hypothesis, there is an appropriate action of $L_5^h(\Gamma)$ on $\cS_\TOP^h(Z)$. This extends, via the bijection $p$, to an action of $L_5^h(\Gamma)$ on $\cS_\TOP^h(pZ,S^3)$. For any orientation character $\omega$, there is a unique $k \geq 0$ such that Wall's exact sequence becomes
\[
0 \longra L_5^h(\Gamma) \xra{~\iota~} L_5^h(\Gamma,1) \longra k\Z = \Ker(\varepsilon) \longra 0.
\]
(Here $k=0$ if and only if $\omega=1$, equivalently, $Z$ is orientable.) Since these groups are abelian, we obtain a non-canonical isomorphism
\[
\varphi: L_5^h(\Gamma,1) \longra L_5^h(\Gamma) \oplus k\Z.
\]
The relevant action of $L_4^h(1)$ on the homology structure set $\cS_\TOP^{h\Z}(S^3)$ via twice-punctured $E_8$-manifolds restricts/extends to an action of $k\Z$ on $\cS_\TOP^h(pZ,S^3)$. Thus, via the isomorphism $\varphi$, we obtain an appropriate action of $L_5^h(\Gamma,1)$, given by concatenation of the actions. Therefore, we obtain $(pZ,S^3)$ has class $SES^h_+$.
\end{proof}

At last, we are ready to establish our main theorem using homology splitting.
For any non-empty compact connected 4-manifold $Z$, we use the following notation:
\begin{eqnarray*}
pZ := Z - \mathrm{int}\,D^4 ~,
& \widetilde{\cN}_\TOP(Z) := \Ker\left(\cN_\TOP(Z) \longra L_4^h(1)\right) ~,
& \widetilde{L}^h_4(\pi_1 Z) := \Cok\left(L^h_4(1) \longra  L^h_4(\pi_1 Z)\right).
\end{eqnarray*}

\begin{proof}[Proof of Theorem~\ref{thm:main}]
Since $\Gamma := \pi_1(X)$ is isomorphic to a free product $\Gamma_1 * \cdots * \Gamma_n$, by an existence theorem of J.~Hillman \cite{HillmanFreeProduct} (cf.~\cite{KLT1, KurMat}), there exist $r \geq 0$ and closed topological 4-manifolds $X_1,\ldots,X_n$ with each $\pi_1(X_i)$ isomorphic to $\Gamma_i$ such that $X \# r(S^2 \times S^2)$ is homeomorphic to $X_1 \# \cdots \# X_n$.  For Part~(1), since each $\Gamma_i$ has class $NDL$, by Theorem~\ref{thm:FTKQ}, we obtain that each $X_i$ has class $SES^h_+$. For Part~(2), this is assumed of the $X_i$, and the $SES^h_+$ property only depends on the homotopy type of $X$. Therefore we may assume that $X = X_1 \# \cdots \# X_n$ with each $X_i$ of class $SES^h_+$. Write $\Gamma_i := \pi_1(X_i)$ for each fundamental group.

We induct on $n > 0$. Assume for some $n \geq 1$ that the $(n-1)$-fold connected sum of all compact connected topological 4-manifolds of class $SES^h_+$ has class $SES^h_+$, where in the non-orientable case we assume 2-torsionfree fundamental group.  Write
\[
X' := X_1 \# \cdots \# X_{n-1}~, \qquad \Gamma' := \Gamma_1 * \cdots * \Gamma_{n-1}.
\]
Hence $X = X' \# X_n$ and $\Gamma = \Gamma' * \Gamma_n$. By hypothesis, both $X'$ and $X_n$ have class $SES^h_+$. Then, by Lemma~\ref{lem:puncture}, the pairs $(pX',S^3)$ and $(pX_n,S^3)$ have class $SES^h_+$. Next, we show our original 4-manifold has class $SES^h_+$:
\[
X = pX' \cup_{S^3} pX_n.
\]

First, the $K$-theory splitting obstruction group vanishes \cite{Waldhausen_Rings}, and, by a recent vanishing result \cite{CR, BL_CAT0, CD2}, so do the $L$-theory obstruction groups:\footnote{If $\Gamma$ is 2-torsionfree, then $\UNil_*^h=0$ by Cappell's earlier result \cite{Cappell_unitary} \cite[Lem.~II.10]{Cappell_split}. Furthermore, we require $\Gamma$ to be 2-torsionfree in the non-orientable case, due to the example of non-vanishing of these two $\UNil$-groups for $\RP^4 \# \RP^4$.}
\begin{eqnarray*}
\widetilde{\Nil}_0(\Z;\Z[\Gamma'-1],\Z[\Gamma_n-1]) = 0,
& \UNil_4^h(\Z;\Z[\Gamma'-1],\Z[\Gamma_n-1]) = 0,
& \UNil_5^h(\Z;\Z[\Gamma'-1],\Z[\Gamma_n-1]) = 0.
\end{eqnarray*}
So observe, by Stalling's theorem for Whitehead groups of free products \cite{Stallings} and the Mayer--Vietoris type exact sequence for $L$-theory groups \cite{Cappell_unitary}, that:
\begin{eqnarray*}
\Wh_1(\Gamma) = \Wh_1(\Gamma') \oplus \Wh_1(\Gamma_n) ~,
& \enspace \widetilde{L}_4^h(\Gamma) = \widetilde{L}_4^h(\Gamma') \oplus \widetilde{L}_4^h(\Gamma_n) ~,
& \enspace \widetilde{L}_5^h(\Gamma) = \widetilde{L}_5^h(\Gamma') \oplus \widetilde{L}_5^h(\Gamma_n) ~.
\end{eqnarray*}
Here, from the Mayer--Vietoris sequence for any free product $G = G_1 * G_2$, we write
\begin{eqnarray*}
\widetilde{L}^h_5(G) &:=& \Ker\left(\bdry: L^h_5(G) \longra L_4^h(1) \right).
\end{eqnarray*}

Second, since $\cN_\TOP(S^3)$ and $\widetilde{\cN}_\TOP(S^3 \times I)$ are singletons, by $\TOP$ transversality \cite{FQ} and by attaching thickened normal bordisms, we obtain:
\begin{eqnarray*}
\widetilde{\cN}_\TOP(X) &=& \widetilde{\cN}_\TOP(X') \times \widetilde{\cN}_\TOP(X_n).
\end{eqnarray*}
So, since the surgery sequence for both $X'$ and $X_n$ is exact at $\cN_\TOP$, the surgery sequence for the connected sum $X$ is exact at $\cN_\TOP$.

Third, since $(pX', S^3)$ and $(pX_n, S^3)$ have class $SES^h_+$ and the splitting obstruction groups vanish, by Theorem~\ref{thm:generalized_weinberger}(1), any homotopy equivalence to $X$ is $\TOP$ $s$-bordant rel $\bdry M$ to a $\Z$-homology split map along $S^3$. That is, the top part of the $s$-bordism is a homotopy equivalence whose preimage of $S^3$ is a $\Z$-homology 3-sphere $\Sigma$. Thus the following inclusion is an equality (compare \cite[Thm.~3]{Cappell_free}):
\[
\subseteq~:~ \cS_\TOP^{\text{$\Z$-split}}(X;S^3) \longra \cS_\TOP^h(X).
\]
By \cite[Corollary~9.3C]{FQ}, there exists a $\TOP$ $\Z$-homology $h$-cobordism $(W;\Sigma,S^3)$ such that $W$ is 1-connected. Furthermore, there exists an extension of the degree one normal map $\Sigma \to S^3$ to a degree one normal map $W \to S^3 \times I$. Thus, by attaching the thickened normal bordism, the following inclusion is an equality:
\[
\subseteq~:~ \cS_\TOP^{\text{split}}(X;S^3) \longra \cS_\TOP^{\text{$\Z$-split}}(X;S^3).
\]
(The process of this last equality is called \emph{neck exchange}, cf.~\cite{KLT2, JK_RP4}.)
Therefore the following map $\#$, given by interior connected sum, is surjective:
\[
\#~:~ \cS_\TOP^h(X') \times \cS_\TOP^h(X_n) \longra \cS_\TOP^h(X).
\]

In order to show that $\#$ is injective, suppose $h_1 \# h_2$ is $\TOP$ $h$-bordant to $h'_1 \# h'_2$, say by a map $H: W \to X \times I$. Since $S^3 \times I$ is a 1-connected 4-manifold \cite{FQ}, and $\bdry H$ is split along $S^3 \times \bdry I$, by the relative 5-dimensional form of Cappell's nilpotent normal cobordism construction \cite{Cappell_free, Cappell_split}, there exists a $\TOP$ normal bordism rel $\bdry W$ from $H$ to an $h$-bordism $H': W' \to X \times I$ split along $S^3 \times I$. So $H' = H'_1 \# H'_2$. Therefore $\#$ is injective.
Now $\Wh_1(\Gamma)$ and $\widetilde{L}_5^h(\Gamma)$ can be given product actions on $\cS_\TOP^h(X)$. The latter extends to an action of $L_5^h(\Gamma)$ by attaching a thickened multiple of a twice-punctured $E_8$ manifold along $S^3$.
Hence the surgery sequence for $X$ is exact at $\cS^h_\TOP$ and $L_5^h$. This completes the induction. Therefore arbitrary connected sums $X=X_1 \# \cdots \# X_n$ have class $SES^h_+$.
\end{proof}

The following argument is partly based on Farrell's 1970 ICM address \cite{Farrell_ICM1970}.

\begin{proof}[Proof of Theorem~\ref{thm:fibering}]
One repeats the mapping torus argument of the proof of \cite[Theorem~5.6]{Khan_fibering}, constructing a homotopy equivalence $h: X \to X$ using $f$. Since the achieved homotopy equivalence $g: M \to X \rtimes_h S^1$ has Whitehead torsion $\tau(g) = \tau(f) = 0$, there are no splitting obstructions. Since $X$ has class $SES^h_+$, the proof of splitting $g$ along $X$ holds \cite[Thm.~5.4]{Khan_fibering}; one no longer requires that $M$ and $X$ be $\DIFF$ manifolds. Therefore the argument of \cite[Theorem~5.6]{Khan_fibering} shows that $f: M \to S^1$ is homotopic to a $\TOP$ $s$-block bundle projection.
\end{proof}

\section{Proofs for topological rigidity}\label{sec:rigidity_proofs}

The following elementary argument is similar to J.~Hillman's \cite[Cor.~6.7.2]{HillmanBook}.

\begin{proof}[Proof of Theorem~\ref{thm:rigidity}]
First, we show that the $s$-cobordism structure set $\cS_\TOP^s(Z)$ is a singleton.  Let $M$ be a compact topological 4-manifold, and let $h: M \to Z$ be a simple homotopy equivalence that restricts to a homeomorphism $\bdry h: \bdry M \to \bdry Z$. Then the surgery obstruction $\sigma_4^s(\eta(h)) \in L_4^s(\pi,\omega)$ vanishes. Since $\sigma_4^s$ is injective, there exists a $\TOP$ normal bordism $F: W \to Z \times I$ to $\eta(h)$ from the identity $\id_Z$. Since $\sigma_5^s$ is surjective, there exists a $\TOP$ normal bordism $F': W' \to Z \times I$ to $\id_Z$ from $\id_Z$ with opposite surgery obstruction: $\sigma_5^s(F') = -\sigma_5^s(F)$. Hence the union
\[
F'' ~:=~ F' \cup_{\id_Z} F ~:~ W' \cup_Z W  \longra Z \times I
\]
is a $\TOP$ normal bordism to $\eta(h)$ from $\id_Z$ with vanishing surgery obstruction: $\sigma_5^s(F'')=0$. Therefore, by 5-dimensional $\TOP$ surgery theory \cite{Wall, KS}, we obtain that $F''$ is $\TOP$ normally bordant $\!\!\rel \bdry$ to a simple homotopy equivalence $F''': (W''';Z,M) \to (Z \times I; Z\times 0, Z \times 1)$ of manifold triads. Therefore we have found a $\TOP$ $s$-bordism to $h$ from $\id_Z$. That is, $\cS_\TOP^s(Z)$ is a singleton $\{*\}$.

Next, observe that trivially we obtain an exact sequence of based sets:
\[
\cN_\TOP(Z \times I) \xrightarrow{~\sigma_5^s~} L_5^s(\pi,\omega) \xrightarrow{~\bdry~} \{*\} \xrightarrow{~\eta~} \cN_\TOP(Z) \xrightarrow{~\sigma_4^s~} L_4^s(\pi,\omega).
\]
We declare the action of $L_5^s(\pi,\omega)$ on $\cS_\TOP^s(Z)$ to be trivial.
Finally, if $\Wh_1(\pi)=0$, then homotopy equivalences to $Z$ are simple, and so $Z$ is topologically $s$-rigid.
\end{proof}

We employ a case of a lemma of Hillman \cite[Lem.~6.8]{HillmanBook}, providing its details.

\begin{proof}[Proof of Corollary~\ref{cor:unstable_rigidity}]
Let $k \geq 0$. By the Mayer--Vietoris sequence in homology, the Shaneson sequence in $L$-theory \cite{ShanesonSequence}, and the Ranicki assembly map \cite[p.~148]{RanickiTOP}, the following diagram commutes with right-split exact rows:
\[\begin{CD}
H_{5+k}(Z;\bL_0) @>{i_*}>> H_{5+k}(Z \times S^1;\bL_0) @>{\bdry}>> H_{4+k}(Z;\bL_0)\\
@VV{A_{5+k}^s(Z)}V @VV{A_{5+k}^s(Z \times S^1)}V @VV{A_{4+k}^h(Z)}V\\
L_{5+k}^s(Z) @>{i_*}>> L_{5+k}^s(Z \times S^1) @>{\bdry}>> L_{4+k}^h(Z).
\end{CD}\]
Moreover, the algebraic right-splitting is given by multiplying local or global quadratic complexes by the symmetric complex of the circle. This choice of splitting commutes with the connective assembly maps $A_{5+k}^s(Z \times S^1)$ and $A_{4+k}^h(Z)$.

Assume $Z \times S^1$ is topologically rigid.  Then $\cS_\TOP^s(Z \times S^1)=\{*\}$. So, by Wall's surgery exact sequence \cite[\S10]{Wall} and Ranicki's identification of the surgery obstruction map with the assembly map \cite[Prop.~18.3(1)]{RanickiTOP} via topological transversality \cite{FQ}, we obtain that $A_{5}^s(Z \times S^1)$ is injective and $A_{6}^s(Z \times S^1)$ is surjective. Hence, using $k=0$ in the above diagram and the right-splitting, $\sigma_4^h = A_4^h(Z)$ is injective. Also, using $k=1$ in the above diagram, $\sigma_{5}^h = A_{5}^h(Z)$ is surjective. Therefore, by Theorem~\ref{thm:rigidity}, we obtain that $\cS_\TOP^h(Z)=\{*\}$. Hence, since $\Wh_1(\pi_1 Z)=0$ by hypothesis, we conclude that $Z$ is topologically $s$-rigid.
\end{proof}

\begin{proof}[Proof of Corollary~\ref{cor:algebraic_rigidity}]
Denote $\Gamma := \pi_1(Z)$. Via topological transversality \cite{FQ}, there are commutative squares with bijective left vertical maps \cite[Prop.~18.3(1)]{RanickiTOP}:
\[\begin{CD}
\cN_\TOP(Z) @>{\sigma_4^s}>> L_4^s(\Gamma)\\
@VV{\cap [Z]_{\bL^0}}V @A{A_4^\Gamma}AA\\
H_4(Z;\bL_0) @>{u_4}>> H_4(B\Gamma;\bL_0)
\end{CD}
\qquad
\begin{CD}
\cN_\TOP(Z \times I) @>{\sigma_5^s}>> L_5^s(\Gamma)\\
@VV{\cap [Z]_{\bL^0}}V @A{A_5^\Gamma}AA\\
H_5(Z;\bL_0) @>{u_5}>> H_5(B\Gamma;\bL_0)
\end{CD}
\]
Here, we are using the identification $\cN_\TOP(Z)=[Z/\bdry Z,G/TOP]_+$. Since $Z$ is aspherical, the bottom horizontal maps are isomorphisms. Since $\Gamma$ is torsionfree with $\cdim(\Gamma)=4$ and has class $FJ_L$, $\Wh_1(\Gamma)=0$, the map $A_4^\Gamma$ is a monomorphism, and $A_5^\Gamma$ is an isomorphism. Hence $\sigma_4^s$ is injective and $\sigma_4^s$ is surjective. Therefore, by Theorem~\ref{thm:rigidity}, we obtain that $Z$ is topologically $s$-rigid and has class $SES^h_+$.
\end{proof}

\begin{proof}[Proof of Corollary~\ref{cor:mappingtorus}]
Let $\alpha: K \to K$ be the homeomorphism.
It follows from the homotopy sequence of a fibration that $Z = K \rtimes_\alpha S^1$ is aspherical. By a recent theorem\footnote{Their proof depends on G.~Perelman's affirmation of Thurston's Geometrization Conjecture (cf.~\cite{AndersonMT_Survey}).  It also depends on individual casework of S.~Roushon and P.~K\"uhl.} of Bartels--Farrell--L\"uck \cite{BFL_lattice}, we obtain that $\Gamma_0 := \pi_1(K)$ has class $FJ_L$.

Write $\Gamma := \pi_1(Z)$. Then $\Gamma = \Gamma_0 \rtimes_\alpha \Z$. By the excisive Wang sequence and the Shaneson Wang-type sequence, there is a commutative diagram with exact rows:
\[\begin{CD}
H_n(B\Gamma_0;\bL) @>{1-\alpha_*}>> H_n(B\Gamma_0; \bL) @>{i_*}>> H_n(B\Gamma; \bL) @>{\bdry}>> H_{n-1}(B\Gamma_0;\bL)\\
@VV{A_n^{\Gamma_0}}V @VV{A_n^{\Gamma_0}}V  @VV{A_n^{\Gamma}}V @VV{A_{n-1}^{\Gamma_0}}V\\
L_n^{s=\infdec}(\Gamma_0) @>{1-\alpha_*}>> L_n^s(\Gamma_0) @>{i_*}>> L_n^s(\Gamma) @>{\bdry}>> L_{n-1}^{h=s}(\Gamma_0)
\end{CD}\]
Since $\Gamma_0$ is torsionfree and has class $FJ_L$, the non-connective assembly maps $A_*^{\Gamma_0}$ are isomorphisms. Hence, by the five lemma, the non-connective assembly maps $A_*^\Gamma$ are isomorphisms. Using topological tranversality and Poincar\'e duality, similar to the proof of Corollary~\ref{cor:algebraic_rigidity}, by Theorem~\ref{thm:rigidity}, we obtain that $\cS_\TOP^s(Z) = \{*\}$.
Hence, since $\Wh_1(\pi_1 Z)=0$, we conclude that $Z$ is topologically $s$-rigid.
\end{proof}

\begin{proof}[Proof of Corollary~\ref{cor:connectsum_vanishingsecondhomotopy}]
Since each $X_i$ is orientable and has class $SES^h_+$, by Theorem~\ref{thm:main}, we obtain that $X$ has class $SES^h_+$ and the following function is a bijection:
\[
\#~:~ \prod_{i=1}^n \cS_\TOP^h(X_i) \longra \cS_\TOP^h(X).
\]
Next, let $1 \leq i \leq n$.  Consider the connective assembly map components \cite{TW}:
\begin{eqnarray*}
A_4 = \left(\begin{smallmatrix}I_0 & \kappa_2\end{smallmatrix}\right) &:&
H_4(B\Gamma_i; \bL_0) = H_0(B\Gamma_i;\Z) \oplus H_2(B\Gamma_i;\Z_2) \longra L_4^h(\Gamma_i)\\
A_5 = \left(\begin{smallmatrix}I_1 & \kappa_3\end{smallmatrix}\right) &:&
H_5(B\Gamma_i; \bL_0) = H_1(B\Gamma_i;\Z) \oplus H_3(B\Gamma_i;\Z_2) \longra L_5^h(\Gamma_i).
\end{eqnarray*}
Assume $\Gamma_i$ is torsionfree and $\pi_2(X_i)\otimes \Z_2 = 0$. Since $\Gamma_i$ has class $FJ_L$ and $\cdim(\Gamma_i) \leq 4$, we obtain that $A_4$ is a monomorphism and $A_5$ is an isomorphism.
Recall the universal covering $\widetilde{X}_i \to X_i$ is classified by a unique homotopy class of map $u: X_i \to B\Gamma_i$, which induces an isomorphism on fundamental groups.
Since $X_i$ is a closed oriented topological manifold, using topological transversality \cite{FQ}, the Quinn--Ranicki $H$-space structure on $G/TOP$, and Poincar\'e duality with respect to the $\bL^0$-orientation \cite{RanickiTOP}, we obtain induced homomorphisms
\begin{eqnarray*}
u_4' &:& \cN_\TOP(X_i) \cong [(X_i)_+, G/TOP]_+ \cong H_4(X_i;\bL_0) \xra{~u_*~} H_4(B\Gamma_i;\bL_0)\\
u_5' &:& \cN_\TOP(X_i \times I) \cong [(X_i)_+ \wedge S_1,G/TOP]_+ \cong H_5(X_i;\bL_0) \xra{~u_*~} H_5(B\Gamma_i;\bL_0)
\end{eqnarray*}
such that the surgery obstruction map factors: $\sigma_4^h  = A_4 \circ u_4'$ and $\sigma_5^h = A_5 \circ u_5'$. Recall the Hopf sequence, which is obtained from the Leray--Serre spectral sequence:
\[
H_3(X_i;\Z_2) \xra{u_3} H_3(B\Gamma_i;\Z_2) \longra H_2(\widetilde{X};\Z_2) \longra H_2(X;\Z_2) \xra{u_2} H_2(B\Gamma_i;\Z_2) \longra 0.
\]
Since $H_2(\widetilde{X};\Z_2) = \pi_2(X_i) \otimes \Z_2=0$, we have $\Ker(u_2)=0$ and $\Cok(u_3)=0$.
Hence
\begin{eqnarray*}
\Ker(\sigma_4^h) &=& \Ker(u_4') = \Ker(u_0) \oplus \Ker(u_2) = 0\\
\Cok(\sigma_5^h) &=& \Cok(u_5') = \Cok(u_1) \oplus \Cok(u_3) = 0.                                                                                                                                                                                                                                                                                                                                                                                                                                                \end{eqnarray*}
Therefore, since $X_i$ has class $SES^h_+$ and $\Wh_1(\Gamma_i)=0$, we obtain that $\cS_\TOP^s(X_i)$ is a singleton. Thus, since $\#$ is a bijection, the Whitehead group $\Wh_1(\Gamma)$ and $s$-cobordism structure set $\cS_\TOP^s(X)$ of $X = X_1 \# \cdots \# X_n$ are singletons also.
\end{proof}


\newpage
\bibliographystyle{elsarticle-harv}
\bibliography{ConnectedSums4Manifolds}

\end{document}